\documentclass[reqno]{amsart}
\usepackage{amscd}
\usepackage{amssymb}
\usepackage[margin=1in,includeheadfoot]{geometry}
\usepackage{cite}
\usepackage [latin1]{inputenc}
\usepackage{tabularx,booktabs,tikz}
\usepackage{caption}
\usepackage{amsmath}
\usepackage{amsfonts}

\usepackage{amscd}
\usepackage{amsthm}
\usepackage{amssymb} \usepackage{latexsym}
\usepackage{eufrak}
\usepackage{euscript}
\usepackage{epsfig}
\usepackage{graphics}
\usepackage{array}
\usepackage{enumerate}
\usepackage{dsfont}
\usepackage{color}
\usepackage{wasysym}
\usepackage{hyperref}
\usepackage{pdfsync}

\def\beq{\begin{equation}}
\def\eeq{\end{equation}}
\def\ba{\begin{array}}
\def\ea{\end{array}}

\newtheorem{thm}{Theorem}[section]
\newtheorem{lm}[thm]{Lemma}
\newtheorem{prop}[thm]{Proposition}
\newtheorem{crl}[thm]{Corollary}

\theoremstyle{definition}
\newtheorem{rem}[thm]{Remark}

\newtheorem{df}[thm]{Definition}

\theoremstyle{remark}

\begin{document}
\pagestyle{plain}
\title{Fractional logarithmic Schr\"{o}dinger equations on lattice graphs}

\author{Lidan Wang}
\email{wanglidan@ujs.edu.cn}
\address{Lidan Wang: School of Mathematical Sciences, Jiangsu University, Zhenjiang 212013, People's Republic of China}

\thanks{
The author is supported by NSFC, no.12401135.}

\begin{abstract}
In this paper, we study the fractional logarithmic Schr\"{o}dinger equation
$$
(-\Delta)^{s} u+h(x) u=u \log u^{2}
$$
on lattice graphs $\mathbb{Z}^d$, where $s\in (0,1)$. If $h(x)$ is a bounded periodic potential, we prove the existence of ground state solution by mountain pass theorem and Lions lemma. If $h(x)$ is a coercive potential, we show the existence of ground state sign-changing solutions by the method of Nehari manifold.
\end{abstract}

\maketitle


\
\


\section{Introduction}
In recent years, a great deal of mathematical effort has been devoted to the nonlinear logarithmic Schr\"{o}dinger equation
$$
(-\Delta)^{s} u+h(x) u=u \log u^{2}
$$
on Euclidean space $\mathbb{R}^d$. In particular, for $s=1$, the classical logarithmic Schr\"{o}dinger equation has been studied extensively in the literature, see for examples \cite{AD,AJ1,FT,J0,S0}. For $s\in(0,1)$, this equation becomes the fractional logarithmic Schr\"{o}dinger equation, we refer the readers to \cite{A,DS,LP,JX,WF}.

Nowadays, people paid much attention to the differential equations on graphs, see for examples\cite{GLY,HL0,HLW,W2,W3,ZZ}. In particular, for the discrete nonlinear logarithmic Schr\"{o}dinger equation, we refer the readers to \cite{CWY,CR,HJ,OZ,SY}. It is worth noting that Chang et al. \cite{CR} proved the existence of ground state sign-changing solutions to the classical logarithmic Schr\"{o}dinger equation. 

Recently,  a great interest has been focused on the study of problems involving the fractional Laplace
operators on graphs. Similar to the continuous case, the discrete 
fractional Laplace operators $(-\Delta)^{s}$ can also be defined in several ways, we refer the readers to \cite{CR0,LM,LR,W0,ZLY}. Based on the definition of pointwise nonlocal formula for $(-\Delta)^{s}$, Xiang and Zhang \cite{XZ} proved the existence of nontrivial homoclinic solutions to a fractional Schr\"{o}dinger equations on lattice graph $\mathbb{Z}$. By the definition through Fourier transform, Wang \cite{W1} established the existence of ground sate solutions to a class of fractional Schr\"{o}dinger equations on lattice graph $\mathbb{Z}^d$. Very recently, Zhang, Lin and Yang \cite{ZLY} obtained the existence of ground state solutions to a fractional Schr\"{o}dinger equation on general weighted graphs by a new definition of fractional Laplace operator.

To the best of our knowledge, there is no existence results for the fractional logarithmic  Schr\"{o}diner equations on graphs. Motivated by the works mentioned above, in this paper, we consider the discrete fractional logarithmic  equation and study the existence of ground state (sign-changing) solutions under different assumptions on potential function.

Let us first give some notations. Let $G=(V, E)$ be a connected, locally finite graph, where $V$ denotes the vertex set and $E$ denotes the edge set. We call vertices $x$ and $y$ neighbors, denoted by $x \sim y$, if there exists an edge connecting them, i.e. $(x, y) \in E$. For any $x,y\in V$, the distance $d(x,y)$ is defined as the minimum number of edges connecting $x$ and $y$, namely
$$d(x,y)=\inf\{k:x=x_0\sim\cdots\sim x_k=y\}.$$
In the following, we consider, the natural discrete model of the Euclidean space, the integer lattice graph.  The $d$-dimensional integer lattice graph, denoted by $\mathbb{Z}^d$, consists of the set of vertices $V=\mathbb{Z}^d$ and the set of edges $E=\{(x,y): x,\,y\in \mathbb{Z}^d,\,\underset {{i=1}}{\overset{d}{\sum}}|x_{i}-y_{i}|=1\}.$
In the sequel, we denote $|x-y|:=d(x,y)$ on the lattice graph $\mathbb{Z}^d$. Let $C(\mathbb{Z}^{d})$ be the set of all functions on $\mathbb{Z}^{d}$ and $C_{c}(\mathbb{Z}^{d})$ the set of all functions with finite support on $\mathbb{Z}^{d}$.  For $p\in [1,+\infty)$,  $\ell^p(\mathbb{Z}^{d})$ is a linear function space with finite norm $
\|u\|_{p}= (\sum\limits_{x \in \mathbb{Z}^d}|u(x)|^{p})^{\frac{1}{p}}$, and $\ell^\infty(\mathbb{Z}^{d})$ is a linear space of all bounded functions with finite norm $\|u\|_\infty= \sup\limits_{x \in \mathbb{Z}^d}|u(x)|.$

In this paper, we study the fractional logarithmic equation
\begin{equation}\label{0.2}
(-\Delta)^{s} u+h(x)u=u \log u^{2}
\end{equation}
on lattice graphs $\mathbb{Z}^d$. Here the fractional Laplace operator $(-\Delta)^{s}$ with $s\in (0,1)$ is given by, see \cite{W0,ZLY},
$$(-\Delta)^{s} u(x)=\sum_{\substack{y\in\mathbb{Z}^d, y\neq x}} w_{s}(x, y)(u(x)-u(y)),\quad u\in\ell^{\infty}(\mathbb{Z}^d),$$
where $w_{s}(x, y)$ is a symmetric positive function satisfying
\begin{equation}\label{28}
c_{s,d}|x-y|^{-d-2s}\leq w_s(x,y)\leq C_{s,d}|x-y|^{-d-2s},\quad x\neq y.
\end{equation}
We always assume that the potential $h$ satisfies 
\begin{itemize}
    \item[($h_1$)] there exists $h_0\in(-1,0)$ such that $\inf\limits_{x\in\mathbb{Z}^d}h(x)\geq h_0;$
   \item[($h_2$)] $h(x)$ is a bounded periodic function for $x\in \mathbb{Z}^d;$
\item[($h_3$)] $h(x)\rightarrow+\infty$ as $|x-x_0|\rightarrow+\infty$ for some $x_0\in\mathbb{Z}^d$.
\end{itemize}

Before we introduce a fractional Sobolev space, we give the definition of fractional gradient form. For $u,v \in \ell^{\infty}(\mathbb{Z}^d)$, the fractional gradient form is given by
\begin{equation*}
\nabla^{s} u \nabla^{s} v(x)=\frac{1}{2} \sum_{y \in \mathbb{Z}^d, y \neq x} w_{s}(x, y)\left(u(x)-u(y)\right)\left(v(x)-v(y)\right). 
\end{equation*}
We denote the length of the fractional gradient as $$\left|\nabla^{s} u(x)\right|=\left(\frac{1}{2} \sum\limits_{y \in \mathbb{Z}^d, y \neq x}w_{s}(x, y)(u(x)-u(y))^{2}\right)^{\frac{1}{2}}.$$ 
For convenience, for $u\in C(\mathbb{Z}^{d})$, we always write $
\int_{\mathbb{Z}^{d}} u(x)\,d \mu=\sum\limits_{x \in \mathbb{Z}^{d}} u(x),
$
where $\mu$ is the counting measure on $\mathbb{Z}^{d}$.

Now we give a fractional Sobolev space
$$W^{s,2}(\mathbb{Z}^d)=\{u\in\ell^\infty(\mathbb{Z}^d): \int_{\mathbb{Z}^d}\left(|\nabla^s u|^{2}+u^{2}\right) d \mu<+\infty\}$$
equipped with the norm
$$
\|u\|_{s,2}=\left(\int_{\mathbb{Z}^d}\left(|\nabla^s u|^{2}+u^{2}\right) d \mu\right)^{\frac{1}{2}}.
$$
Then by \cite[Theorem 2.3]{ZLY}, one gets that the space $W^{s,2}(\mathbb{Z}^d)$ is a Hilbert space with the inner product
$$
(u, v)_{s,2}=\int_{\mathbb{Z}^d}(\nabla^s u \nabla^s v+u v)\,d \mu.
$$
Under the assumption $(h_1)$, we define a new subspace
$$
\mathcal{H}_{s}=\left\{u \in W^{s,2}(\mathbb{Z}^d): \int_{\mathbb{Z}^d} (h(x)+1) u^{2}\,d \mu<\infty\right\}
$$
under the norm
$$
\|u\|_{\mathcal{H}_{s}}^{2}=\int_{\mathbb{Z}^d}\left(|\nabla^s u|^{2}+(h(x)+1)u^{2}\right)\,d \mu,
$$
which is induced by the inner product
$$
(u, v)_{\mathcal{H}_{s}}=\int_{\mathbb{Z}^d}(\nabla^s u \nabla^s v+(h(x)+1)u v)\,d \mu.
$$
Clearly, the space $\mathcal{H}_{s}$ is also a Hilbert space. By $(h_1)$, we have that $$\int_{\mathbb{Z}^d}u^2\,d\mu\leq \frac{1}{(h_0+1)}\int_{\mathbb{Z}^d}(h(x)+1)u^2\,d\mu,$$
and hence, for $q\geq 2$,
\begin{equation}\label{ac}
    \|u\|_q\leq \|u\|_2\leq \|u\|_{s,2}\leq \frac{1}{\sqrt{(h_0+1)}}\|u\|_{\mathcal{H}_s}.
\end{equation}

\
\

The energy functional $J:\mathcal{H}_s\rightarrow\mathbb{R}\cup \{+\infty\}$ related to the equation (\ref{0.2}) is given by
\begin{equation*}
J(u)=\frac{1}{2} \int_{\mathbb{Z}^d}\left(|\nabla^s u|^{2}+(h(x)+1)u^{2}\right) d \mu-\frac{1}{2} \int_{\mathbb{Z}^d}u^{2} \log u^{2} d \mu. 
\end{equation*}
Since there exists $u\in W^{s,2}(\mathbb{Z}^d)$ such that $\int_{\mathbb{Z}^d}u^{2} \log u^{2} d \mu=-\infty,$ see "Appendix" below, the functional $J$ fails to be $C^1$ on $\mathcal{H}_s.$ Therefore, we consider the functional $J$ on the set
$$D(J)=\{u\in\mathcal{H}_s: \int_{\mathbb{Z}^d}u^{2} |\log u^{2}|\,d \mu <+\infty\}.$$
Given $u\in D(J)$, for any $\phi\in D(J)$, one can get easily that
$$\langle J'(u),\phi\rangle=\int_{\mathbb{Z}^d}\left( \nabla^s u \nabla^s \phi+h(x)u \phi\right)\,d \mu-\int_{\mathbb{Z}^d} u\phi\log u^2\,d\mu.$$
We say that $u\in D(J)$ is a weak solution of the equation (\ref{0.2}) if $u$ is a critical point of $J$, i.e. $\langle J'(u),\phi\rangle=0$ for all $\phi\in D(J)$.
Moreover, we define the Nehari manifold $$\mathcal{N}=\left\{u \in D(J)\backslash\{0\}: \langle J^{\prime}(u),u\rangle=0\right\}$$
and the sign-changing Nehari manifold
$$\mathcal{M}=\left\{u \in D(J): u^{ \pm} \neq 0 \text { and } \langle J^{\prime}(u), u^{+}\rangle=\langle J^{\prime}(u),u^{-}\rangle=0\right\},$$
where $u^{+}=\max \{u, 0\}$ and $u^{-}=\min \{u, 0\}.$
Let
$$c=\inf\limits_{u\in\mathcal{N}}J(u), \quad \text{and}\quad m=\inf\limits_{u\in \mathcal{M}}J(u).$$
We say that $u\in D(J)$ is a ground state solution if
$J(u)=c$ and a ground state sign-changing solution if $J(u)=m.$

Our main results are as follows.

\begin{thm}\label{t1}
Let $(h_1)$ and $(h_2)$ hold. Then the equation (\ref{0.2}) has a ground state solution.
\end{thm}

\begin{thm}\label{t2}
Let $(h_1)$ and $(h_3)$ hold. Then the equation (\ref{0.2}) has a ground state sign-changing solution.
\end{thm}

\begin{rem}
 One of the main challenges in proving Theorem \ref{t1}  and Theorem \ref{t2} is to handle the logarithmic term in the equation (\ref{0.2}). In the continuous case, the fractional logarithmic Sobolev inequality plays a key role in studying fractional logarithmic Schr\"{o}dinger equations, see \cite{HL1,LP}. In the discrete case, our arguments do not rely on the fractional logarithmic Sobolev inequality.
\end{rem}

This paper is organized as follows. In Section 2, we state some basic results on graphs. In Section 3, we prove the existence of ground state solutions to the equation (\ref{0.2})(Theorem \ref{t1}). In Section 4, we show the existence of ground state sign-changing solutions to the equation (\ref{0.2})(Theorem \ref{t2}).

\section{Preliminaries}
In this section, we give some basic lemmas that useful in this paper.

\begin{lm}\label{lb}
 The space $W^{s,2}(V)$ and $\ell^2(V)$  are equivalent. 
\end{lm}
\begin{proof}
Clearly, for $u\in W^{s,2}(V)$, we have that $u\in\ell^2(V),$ namely
$$\|u\|_2\leq \|u\|_{s,2}<+\infty.$$

On the other hand, for $u\in\ell^2(V)$, by (\ref{28}), we get that
 $$
\begin{aligned}
    \int_{\mathbb{Z}^d} |\nabla^s u|^2\,d\mu=&\frac{1}{2}\sum\limits_{x\in\mathbb{Z}^d} \sum\limits_{y \in \mathbb{Z}^d, y \neq x}w_{s}(x, y)(u(y)-u(x))^{2}\\ \leq& C_{s,d}\sum\limits_{x\in\mathbb{Z}^d} \sum\limits_{y \in \mathbb{Z}^d, y \neq x}\frac{|u(x)|^2+|u(y)|^2}{|x-y|^{d+2s}}\\=&C_{s,d}\sum\limits_{z\neq 0}\frac{|u(0)|^2+|u(z)|^2}{|z|^{d+2s}}+C_{s,d}\sum\limits_{x\neq 0} \sum\limits_{z\neq 0}\frac{|u(x)|^2+|u(x+z)|^2}{|z|^{d+2s}}\\ \leq &3 C_{s,d}\sum\limits_{z\neq 0}\frac{1}{|z|^{d+2s}}\sum\limits_{x\in\mathbb{Z}^d}|u(x)|^2\\ \leq &C\sum\limits_{x\in\mathbb{Z}^d}|u(x)|^2\\ <&+\infty,
\end{aligned}
$$
where $\sum\limits_{z\neq 0}\frac{1}{|z|^{d+2s}}<+\infty.$ Thus $u\in W^{s,2}(V)$ and
$\|u\|_{s,2}\leq \sqrt{C+1}\|u\|_2.$

Therefore we have that
$$\|u\|_2\leq \|u\|_{s,2}\leq \sqrt{C+1}\|u\|_2.$$
Note that $C_c(V)$ is dense in $W^{s,2}(V)$ and $\ell^2(V)$. Then we obtain that $W^{s,2}(V)$ and $\ell^2(V)$ are equivalent.
\end{proof}

\begin{lm}\label{lp}
 Let $(h_1)$ and $(h_2)$ hold. Then $\mathcal{H}_s(V)$ and $\ell^2(V)$ are equivalent.  \end{lm}
 
 \begin{proof}
 We only need to prove the equivalence between $W^{s,2}(V)$ and $\mathcal{H}_s(V)$. In fact, by $(h_1)$ and $(h_2)$, we have
 $$(h_0+1)\leq (h(x)+1)\leq C+1,$$
 where $C$ is a positive constant. Then we get that
 $$\int_Vu^2\,d\mu \leq \frac{1}{(h_0+1)}\int_V (h(x)+1)u^2\,d\mu\leq (C+1)\int_V u^2\,d\mu,$$
 and hence
 $$\|u\|^2_{s,2}\leq \frac{1}{(h_0+1)}\|u\|^2_{\mathcal{H}_s}\leq (C+1)\|u\|^2_{s,2}.$$
This means that $W^{s,2}(V)$ and $\mathcal{H}_s(V)$ are equivalent. By Lemma \ref{lb}, we get that  $\mathcal{H}_s(V)$ and $\ell^2(V)$ are equivalent.
 \end{proof}

\
\

Since the definition of fractional Laplacian in \cite{ZLY} is a generalization of that in \cite{W0}, the formulas of integration by parts in \cite{ZLY} also hold on lattice graphs $V$. 

\begin{lm}\label{l0} 
Let $u \in W^{s, 2}(V)$. Then for any $\phi\in C_{c}(V)$, we have
$$
\int_{V}\nabla^s u \nabla^s\phi \,d \mu=\int_{V}(-\Delta)^s u \phi\,d\mu.
$$
\end{lm}

\begin{lm}\label{l2.1}
  If $u \in D(J)$ is a weak solution to the equation (\ref{0.2}), then $u$ is a pointwise solution to the equation (\ref{0.2}).  
\end{lm}

\begin{proof}
Since $u \in D(J)$ is a weak solution of the equation (\ref{0.2}), for any $\varphi \in D(J)$, 
$$
\int_{V}(\nabla^su\nabla^s \phi+h(x) u \phi)\,d \mu=\int_{V} u\phi\log u^2\,d\mu.
$$
Note that $C_c(V)$ is dense in $D(J)$. By Lemma \ref{l0}, we have
\begin{equation}\label{2.1}
\int_{V}(-\Delta)^s u\,\phi\,d\mu+\int_{V}h(x) u \varphi\,d \mu=\int_{V}u\varphi \log u^{2}\,d\mu,\quad \phi\in C_c(V).
\end{equation}
For any fixed $x_0\in V$, let $\varphi_0(x_0)=1$ and $\varphi_0(x)=0$ for the rest.
Clearly, $\varphi_0 \in C_c(V)$. Substituting $\phi_0$ into (\ref{2.1}), we get that
$$(-\Delta)^s u(x_0)+ h(x_0)        u(x_0)=u(x_0)\log u^{2}(x_0).$$ Thus $u$ is a pointwise solution of the equation (\ref{0.2}) since $x_0$ is arbitrary.

\end{proof}

We give two compactness results for $\mathcal{H}_s$. The first one is a discrete Lions lemma, which plays an important role in the proof of Theorem \ref{t1}.
\begin{lm}\label{li} Let $(h_1)$ hold. For $2\leq q<+\infty$, if $\left\{u_{k}\right\}$ is bounded in $\mathcal{H}_s$ and
$$
\left\|u_{k}\right\|_{\infty} \rightarrow 0,\quad k \rightarrow+\infty,
$$
then, for any $q\in(2,+\infty)$,
$$
u_{k} \rightarrow 0,\quad  \text { in } \ell^{q}(V).
$$   
\end{lm} 

\begin{proof}
 Since $\left\{u_{k}\right\}$ is bounded in $\mathcal{H}_s$, it follows from (\ref{ac}) that $\{u_k\}$ is bounded in $\ell^{2}(V)$. Hence, for $2<q<+\infty$, by an interpolation inequality, we obtain that
$$
\left\|u_{k}\right\|_{q}^{q} \leq\left\|u_{k}\right\|_{2}^{2}\left\|u_{k}\right\|_{\infty}^{q-2}\rightarrow 0,\quad k\rightarrow+\infty.
$$
\end{proof} 

The second one is a compact embedding result, we refer the readers to \cite[Lemma 5.1]{ZLY}.
\begin{lm}\label{l2}
 Let $(h_1)$ and $(h_3)$ hold. Then the space $\mathcal{H}_{s}$ is continuously embedded into $\ell^{q}(V)$ with $q\in[2,+\infty].$ Moreover, for any bounded sequence $\{u_k\}\subset \mathcal{H}_{s}$, there exists $u\in\mathcal{H}_{s}$ such that, up to a subsequence,
 $$
\begin{cases}u_k \rightharpoonup u, & \text { weakly in } \mathcal{H}_{s}, \\ u_k \rightarrow u, & \text { pointwise in } V, \\ u_k \rightarrow u, & \text { strongly in } \ell^{q}(V),~q \in[2,+\infty].\end{cases}
$$
\end{lm}

Finally, we give some results about sign-changing solutions related to the equation (\ref{0.2}).
\begin{lm}\label{10}
 Let $u\in \mathcal{H}_s.$ Then $u^+, u^-\in \mathcal{H}_s.$
\end{lm}
\begin{proof}
    We refer the readers to \cite[Lemma 5.2]{ZLY}.
\end{proof}

\begin{prop}\label{o} Let $\alpha,\beta>0$. Then for any $u\in \mathcal{H}_s$, 
  \begin{itemize}
      \item[(i)] $$
      \int_{V} \left|\nabla^s(\alpha u^{+}+\beta u^{-})\right|^2 d \mu=\int_{V} \left|\nabla^s (\alpha u^{+})\right|^2 d \mu+\int_{V}\left|\nabla^s (\beta u^{-})\right|^2 d \mu-\alpha\beta K(u),
      $$
 where $K(u)=\sum\limits_{x \in V} \sum\limits_{y \neq x}w_s(x,y)\left[u^{+}(x) u^{-}(y)+u^{-}(x) u^{+}(y)\right]\leq 0;$    
 \item[(ii)] $$
 \int_{V} \nabla^s\left(\alpha u^{+}+\beta u^{-}\right)\nabla^s (\alpha u^{+})\, d \mu=\int_{V} \left|\nabla^s(\alpha u^{+})\right|^2 d \mu-\frac{\alpha\beta}{2} K(u);
 $$
 \item[(iii)] $$
 \int_{V} \nabla^s\left(\alpha u^{+}+\beta u^{-}\right)\nabla^s (\beta u^{-})\, d \mu=\int_{V} \left|\nabla^s (\beta u^{-})\right|^2 d \mu-\frac{\alpha\beta}{2} K(u).
$$
  \end{itemize}  
\end{prop}

\begin{proof}
(i) A direct calculation yields that
$$
\begin{aligned}
& \int_{V} \left|\nabla^s(\alpha u^{+}+\beta u^{-})\right|^2 d \mu \\
= & \frac{1}{2} \sum_{x \in V} \sum_{y \neq x}w_s(x,y)\left[\left(\alpha u^{+}+\beta u^{-}\right)(x)-\left(\alpha u^{+}+\beta u^{-}\right)(y)\right]^{2}\\
= & \frac{1}{2} \sum_{x \in V} \sum_{y \neq x}w_{s}(x,y)\left[\left(\alpha u^{+}(x)-\alpha u^{+}(y)\right)^{2}+\left(\beta u^{-}(x)-\beta u^{-}(y)\right)^{2}-2\alpha\beta\left[u^{+}(x) u^{-}(y)+u^{-}(x) u^{+}(y)\right]\right] \\
= & \int_{V} \left|\nabla^s (\alpha u^{+})\right|^2 d \mu+\int_{V}\left|\nabla^s (\beta u^{-})\right|^2 d \mu-\alpha\beta K(u).
\end{aligned}
$$

(ii) By a direct computation, we get that
$$
\begin{aligned}
&\int_{V} \nabla^s\left(\alpha u^{+}+\beta u^{-}\right)\nabla^s (\alpha u^{+})\, d \mu \\= & \frac{1}{2} \sum_{x \in V} \sum_{y \neq x}w_s(x,y)\left[\left(\alpha u^{+}+\beta u^{-}\right)(x)-\left(\alpha u^{+}+\beta u^{-}\right)(y)\right]\left[\alpha u^{+}(x)-\alpha u^{+}(y)\right] \\
= & \frac{1}{2} \sum_{x \in V} \sum_{y \neq x}w_s(x,y)\left[\left(\alpha u^{+}(x)-\alpha u^{+}(y)\right)^{2}-\alpha\beta\left(u^{+}(x) u^{-}(y)+u^{-}(x) u^{+}(y)\right)\right] \\
= & \int_{V} \left|\nabla^s(\alpha u^{+})\right|^2 d \mu-\frac{\alpha\beta}{2} K(u) .
\end{aligned}
$$

(iii) The proof is similar to that of (ii), we omit here.

\end{proof}

\begin{crl}\label{co}  Let $\alpha,\beta>0$. Then for any $u \in D(J)$,
\begin{itemize}
    \item[(i)] $$
\begin{aligned}
J(\alpha u^{+}+\beta u^{-})=&J\left(\alpha u^{+}\right)+J\left(\beta u^{-}\right)-\frac{\alpha\beta}{2}K(u);
\end{aligned}
$$

\item[(ii)] $$
\begin{aligned}
\langle J^{\prime}(\alpha u^{+}+\beta u^{-}),\alpha u^{+}\rangle =&\langle J^{\prime}\left(\alpha u^{+}\right),\alpha u^{+}\rangle-\frac{\alpha\beta}{2}K(u);
\end{aligned}
$$

\item[(iii)] $$
\begin{aligned}
\langle J^{\prime}(\alpha u^{+}+\beta u^{-}),\beta u^{-}\rangle &=\langle J^{\prime}\left(\beta u^{-}\right),\beta u^{-}\rangle-\frac{\alpha\beta}{2}K(u).
\end{aligned}
$$
\end{itemize}
\end{crl}

\begin{proof}
For $\alpha,\beta >0$, we have 
$$
\begin{aligned}
&\int_V(\alpha u^{+}+\beta u^{-})^2\log(\alpha u^{+}+\beta u^{-})^2\,d\mu\\=&\int_{\{u\geq 0\}}(\alpha u^{+}+\beta u^{-})^2\log(\alpha u^{+}+\beta u^{-})^2\,d\mu+\int_{\{u<0\}}(\alpha u^{+}+\beta u^{-})^2\log(\alpha u^{+}+\beta u^{-})^2\,d\mu\\=&\int_{\{u\geq 0\}}(\alpha u^{+})^2\log(\alpha u^{+})^2\,d\mu+\int_{\{u<0\}}(\beta u^{-})^2\log(\beta u^{-})^2\,d\mu\\=&\int_{V}
\left[(\alpha u^{+})^2\log(\alpha u^{+})^2+(\beta u^{-})^2\log(\beta u^{-})^2\right]\,d\mu.
\end{aligned}
$$
(i) By Proposition \ref{o}, we obtain that
$$
\begin{aligned}
&J(\alpha u^++\beta u^-)\\=&\frac{1}{2}
\int_{V} |\nabla^s (\alpha u^++\beta u^-)|^2\,d\mu+\frac{1}{2}\int_{V}(h(x)+1)(\alpha u^++\beta u^-)^2\,d \mu\\&-\frac{1}{2}\int_{V}(\alpha u^++\beta u^-)^2\log (\alpha u^++\beta u^-)^2 \,d\mu\\=&\frac{1}{2}\left(\int_{V} \left|\nabla^s(\alpha u^+)\right|^2 d \mu+\int_{V} \left|\nabla^s(\beta u^-)\right|^2 d \mu-\alpha\beta K(u)\right)+\frac{1}{2}\int_{V}(h(x)+1)\left[(\alpha u^+)^2+(\beta u^-)^2\right]\,d \mu\\&-\frac{1}{2}\int_{V}
\left[(\alpha u^{+})^2\log(\alpha u^{+})^2+(\beta u^{-})^2\log(\beta u^{-})^2\right]\,d\mu\\=& \frac{1}{2}\int_{V}  |\nabla^s(\alpha u^+)|^2 +(h(x)+1)(\alpha u^+)^2\,d \mu-\frac{1}{2}\int_{V}
(\alpha u^{+})^2\log(\alpha u^{+})^2\,d\mu\\&+\frac{1}{2}\int_{V}|\nabla^s(\beta u^-)|^2 +(h(x)+1)(\beta u^-)^2\,d \mu-\frac{1}{2}\int_{V}
(\beta u^{-})^2\log(\beta u^{-})^2\,d\mu-\frac{\alpha\beta}{2}K(u)\\=&J\left(\alpha u^+\right)+J(\beta u^-)-\frac{\alpha\beta}{2}K(u).
\end{aligned}
$$

\
\

(ii) It follows from Proposition \ref{o} that
$$
\begin{aligned}
&\langle J^{\prime}(\alpha u^{+}+\beta u^{-}),\alpha u^{+}\rangle\\=&
\int_{V} \nabla^s (\alpha u^{+}+\beta u^{-})\nabla^s (\alpha u^{+})\,d\mu+\int_{V}h(x)(\alpha u^{+}+\beta u^{-})(\alpha u^{+})\,d \mu\\&-\int_{V}(\alpha u^{+}+\beta u^{-})(\alpha u^{+})\log (\alpha u^{+}+\beta u^{-})^2\,d\mu\\=& \left(\int_{V} \left|\nabla^s(\alpha u^{+})\right|^2 d \mu-\frac{\alpha\beta}{2} K(u)\right)+\int_{V}h(x)(\alpha u^{+})^2\,d \mu-\int_{V}(\alpha u^{+})^{2}\log (\alpha u^{+})^2\,d\mu \\=& \int_{V}|\nabla^s(\alpha u^{+})|^2 +h(x)(\alpha u^{+})^2\,d \mu-\int_{V}(\alpha u^{+})^{2}\log (\alpha u^{+})^2\,d\mu-\frac{\alpha\beta}{2} K(u)\\=&\langle J^{\prime}\left(\alpha u^{+}\right),\alpha u^{+}\rangle-\frac{\alpha\beta}{2}K(u).
\end{aligned}
$$

\
\

(iii) The proof is similar to that of (ii), we omit it here.
\end{proof}

\section{Proof of Theorem \ref{t1}}
In this section, under the assumptions $(h_1)$ and $(h_2)$, we prove the existence of ground state solutions to the equation (\ref{0.2}). First, we recall some known results about the lower semicontinuous functional, see \cite{S}.

\begin{df}
 Let $X$ be a real Banach space and $I$ a functional on $X$ such that $I=\Phi+\Psi$, where $\Phi\in C^1(X,\mathbb{R})$ and $\Psi: X\rightarrow\mathbb{R}\cup\{+\infty\} $ with $\Psi\not\equiv+\infty$ is convex and lower semicontinuous. A Palais-Smale sequence at level $d\in\mathbb{R}$, $(PS)_d$ sequence for short, for the functional $I$ is a sequence $\{u_k\}\subset X$ such that $I(u_k)\rightarrow d$ and
\begin{equation}\label{37}
\langle \Phi'(u_k),v-u_k\rangle+\Psi(v)-\Psi(u_k)\geq -\varepsilon_k\|v-u_k\|,\quad v\in X,
\end{equation}
where $\varepsilon_k\rightarrow 0.$ 

An equivalent definition is that a sequence $\{u_k\}\subset X$ such that $I(u_k)\rightarrow d$ and
\begin{equation}\label{36}
  \langle \Phi'(u_k),v-u_k\rangle+\Psi(v)-\Psi(u_k)\geq \langle z_k, v-u_k\rangle,\quad v\in X,
\end{equation}
where $z_k=I'(u_k)\rightarrow 0.$ Moreover, we say that $I$ satisfies $(PS)$ condition if each $(PS)$ sequence has a convergent subsequence.   
\end{df}

The following result is about the mountain pass theorem without $(PS)$ condition, see \cite[Theorem 3.1]{AD}.
\begin{thm}\label{lt}
Let $X$ be a real Banach space and $I: X\rightarrow\mathbb{R}\cup\{+\infty\}$ be a functional such that
\begin{itemize}
    \item[(i)] $I=\Phi+\Psi$, where $\Phi\in C^1(X,\mathbb{R})$ and $\Psi: X\rightarrow\mathbb{R}\cup\{+\infty\} $ with $\Psi\not\equiv+\infty$ is convex and lower semicontinuous;
    \item[(ii)] $I(0)=0$ and $I|_{\partial B_\rho}\geq \alpha$, for some constants $\rho,\alpha>0$;
    \item[(iii)] $I(e)\leq 0$, for some $e\not\in\bar{B}_\rho(0)$. 
\end{itemize}
If $$d_0:=\inf _{\gamma \in \Gamma} \max_{t \in[0,1]} I(\gamma(t)),\qquad
\Gamma=\{\gamma \in C([0,1], X): \gamma(0)=0, I(\gamma(1))<0\},
$$
then for a given $\varepsilon>0$, there is a $u_\varepsilon\in X$ such that
$$I(u_\varepsilon)\in[d_0-\varepsilon,d_0+\varepsilon],$$
and
$$\langle \Phi'(u_\varepsilon),v-u_\varepsilon\rangle+\Psi(v)-\Psi(u_\varepsilon)\geq -3\varepsilon\|v-u_\varepsilon\|,\quad v\in X.$$

\end{thm}

\
\

As we said before, there exists $u\in\mathcal{H}_s(V)$ such that $\int_{V}u^{2} \log u^{2} d \mu=-\infty$. This leads to the lack of smoothness for $J$ on $\mathcal{H}_s.$ In order to overcome this difficulty, we decompose the functional $J$ into a sum of a $C^{1}$ functional and a convex lower semicontinuous functional.

For $t\in \mathbb{R}$ and $\delta>0$, let
$$
A(t)= \begin{cases}
-\frac{1}{2} t^{2} \log t^{2}, & 0\leq|t|\leq \delta, \\ -\frac{1}{2} t^{2}\left(\log \delta^{2}+3\right)+2 \delta|t|-\frac{1}{2} \delta^{2}, & |t| \geq \delta,
\end{cases}
$$
and
$$
B(t)=\frac{1}{2} t^{2} \log t^{2}+A(t).
$$
Then we have the below proposition, see \cite{JS,SS}.
\begin{prop} \label{p1}  For functions $A(t)$ and $B(t)$, we  have the properties: 
\begin{itemize}
\item [(i)] $A(t), B(t) \in C^{1}(\mathbb{R}, \mathbb{R})$;
    \item[(ii)] for $\delta>0$ small, $A(t)$ is a nonnegative convex and even function and satisfies $A^{\prime}(t) t \geq 0 $ with $t \in \mathbb{R}$;
\item[(iii)] for $p \in(2,+\infty)$, there exists $C>0$ such that $\left|B^{\prime}(t)\right| \leq C|t|^{p-1}$ with $t \in \mathbb{R}.$
\end{itemize}
\end{prop}

\
\

Let
$$
J_1(u)=\frac{1}{2} \int_{V}\left(|\nabla^s u|^{2}+(h(x)+1)\right)u^{2}) d \mu-\int_{V} B(u)\,d\mu,
$$
and
$$
J_2(u)=\int_{V}A(u)\,d\mu.
$$
For $u \in\mathcal{H}^s,$ by Proposition \ref{p1}, we have that
\begin{equation*}
J(u)=J_1(u)+J_2(u),
\end{equation*}
where $J_1\in C^1(\mathcal{H}^s,\mathbb{R})$ and $J_2\geq 0$ is convex and lower semicontinuous.

\begin{lm}\label{cm}
Let $\{u_k\}$ be a sequence in $\mathcal{H}_s$.  If $J(u_k)$ is bounded in $\mathcal{H}_s$, then $J'(u_k)\rightarrow 0$ if and only if $\{u_k\}$ is a $(PS)$ sequence for $J.$ 
\end{lm}

\begin{proof}
Since $J=J_1+J_2$, where $J_1\in C^1(\mathcal{H}^s,\mathbb{R})$ and $J_2$ is convex and lower semicontinuous,  the result follows from the equivalence of (\ref{37}) and (\ref{36}).   
\end{proof}

Next, we show that the functional J satisfies the mountain pass geometry.
\begin{lm}\label{lh}
Let $\left(h_{1}\right)$ hold. Then we have
 \begin{itemize}
     \item[(i)]  for any $u \in \mathcal{H}_s$, there exist $\alpha, \rho>0$ such that $J(u) \geq \alpha$  with $\|u\|_{\mathcal{H}_s}=\rho$;
     \item[(ii)] there exists $e \in \mathcal{H}_s$ with $\|e\|>\rho$ such that $J(e)<0$.
 \end{itemize}   
\end{lm}

\begin{proof}
 (i) 
 For $p>2$, by Proposition \ref{p1}, one gets that
 $$|B(t)|\leq C|t|^p.$$
 Moreover, for $t\in \mathbb{R}$, we have $A(t) \geq 0$. Then it follows (\ref{ac}) that
$$
\begin{aligned}
  J(u)=&\frac{1}{2}\|u\|_{\mathcal{H}_s}^{2}+ \int_{V} A(u)\,d\mu-\int_{V}B(u)\,d\mu\\ \geq &\frac{1}{2}\|u\|_{\mathcal{H}_s}^{2}-\int_{V} B(u)\,d \mu \\ \geq &\frac{1}{2}\|u\|_{\mathcal{H}_s}^{2}-C\|u\|_p^{p}\\ \geq &\frac{1}{2}\|u\|_{\mathcal{H}_s}^{2}-C\|u\|_{\mathcal{H}_s}^{p}.
\end{aligned}
$$
Since $p>2$, there exists
$\alpha>0$ and $\rho>0$ such that $\|u\|_{\mathcal{H}_s}=\rho>0$.\\

(ii) Let $u \in D(J) \backslash\{0\}$. As $t \rightarrow+\infty$, one gets that
\begin{equation}\label{ud}
\begin{aligned}
 J(tu)=&\frac{t^{2}}{2} \int_{V}\left(|\nabla^s u|^{2}+(h(x)+1) u^{2}\right)\,d\mu-\frac{1}{2}\int_{V}(tu)^2\log(tu)^2\,d\mu\\=&\frac{t^{2}}{2}\left[\|u\|_{\mathcal{H}_s}^{2}-\int_{V} u^{2} \log u^{2} d \mu-\log t^{2}\int_{V} u^{2}\,d\mu\right]\\ \rightarrow&-\infty.   
\end{aligned}
\end{equation}
Therefore, there exists $t_0>0$ large enough such that $\|t_0 u\|_{\mathcal{H}_s}>\rho$ and $J(t_0 u)<0$. By taking $e=t_0 u$, we get the result.

\end{proof}

\begin{lm}\label{lg}
 Let $\left(h_{1}\right)$ hold. Then 
for any $u \in  D(J)\backslash\{0\}$, there exists a unique $t_{u}>0$ such that $t_{u} u \in \mathcal{N}$ and $J(t_{u} u)=$ $\max\limits_{t>0} J(t u)$.
  
\end{lm}

\begin{proof} Fix $u \in D(J)\backslash\{0\}$. For $t>0$, 
$$
\begin{aligned}
J(tu)=\frac{t^{2}}{2}\left[\|u\|_{\mathcal{H}_s}^{2}-\int_{V} u^{2} \log u^{2} d \mu-\log t^{2}\int_{V} u^{2}\,d\mu\right].
\end{aligned}
$$
One gets easily that $J(t u)>0$ for $t>0$ small enough.
Moreover, by (\ref{ud}), we have that
$ J(tu) \rightarrow-\infty$ as $t\rightarrow+\infty.$
Thus $\max\limits_{t>0} J(t u)$ is achieved at some $t_{u}>0$ with $t_{u} u \in \mathcal{N}$. 

Now we show the uniqueness of $t_{u}$. By contradiction, suppose that there exist $t_{u}^{\prime}>t_{u}>0$ such that $t_{u}^{\prime} u, t_{u} u \in \mathcal{N}$. Then we have
$$
\begin{aligned}
& \|u\|_{\mathcal{H}_s}^2=\log (t_u)^2\int_{V}u^2\,d\mu+\int_{V}u^2\log u^2\,d\mu+\int_{V}u^2\,d\mu, \\
& \|u\|_{\mathcal{H}_s}^2=\log (t'_u)^2\int_{V}u^2\,d\mu+\int_{V}u^2\log u^2\,d\mu+\int_{V}u^2\,d\mu.
\end{aligned}
$$
This implies that
$$
\begin{aligned}
0=\left(\log (t_u)^2-\log (t'_u)^2\right)\int_{V}u^2\,d\mu<0,
\end{aligned}
$$
which is a contradiction.
\end{proof}

\begin{lm}\label{ld}  Let $\left(h_{1}\right)$ and $(h_2)$ hold. Then we have
$$c=\inf_{u \in \mathcal{N}} J(u)=\inf _{\gamma \in \Gamma} \max_{t \in[0,1]} J(\gamma(t))=:c_0>0,$$ 
where $$
\Gamma=\{\gamma \in C([0,1], \mathcal{H}_s): \gamma(0)=0, J(\gamma(1))<0\}.
$$
\end{lm}

\begin{proof} 
By Lemma \ref{lh}, one gest easily that $c_0>0$. We first prove that $c\geq c_0$. In fact, by Lemma \ref{lg}, there exists a unique $t_u>0$ such that
$J(t_{u}u)=\max\limits_{t>0} J(tu)$. Then 
$$\inf\limits_{u\in D(J) \backslash\{0\}} \max _{t>0} J(tu)=\inf\limits_{u\in D(J) \backslash\{0\}}J(t_{u}u)= \inf\limits_{u\in\mathcal{N}} J(u)=c.$$
For $u\in D(J)\backslash\{0\}$, by (\ref{ud}), there exists $t_0>0$ such that $J(t_0u)<0$. Define
$$
\begin{aligned}
\gamma_0: [0,1] & \rightarrow \mathcal{H}_s, \\
 t & \mapsto tt_0u.
\end{aligned}
$$ 
Since $\gamma_0(0)=0$ and $J(\gamma_0(1))<0$, we have $\gamma_0\in\Gamma$. Then we have 
$$\max _{t>0} J(tu)\geq \max _{t\in[0,1]} J(tt_0u)=\max _{t\in[0,1]} J(\gamma_0(t))\geq \inf\limits_{\gamma\in\Gamma}\max\limits_{t\in [0,1]} J(\gamma(t)),$$
which implies that
$c\geq c_0.$

Next, we prove that $c\leq c_0.$  Indeed, by Theorem \ref{lt} and Lemma \ref{lh}, there exists a $(PS)_{c_0}$ sequence $\{u_{k}\} \subset \mathcal{H}_s$ for $J$. By Lemma \ref{cm}, we have that
\begin{equation}\label{29}
J(u_k)\rightarrow c_0,\quad \text{and}\quad J'(u_k)\rightarrow 0,\quad k\rightarrow+\infty.    
\end{equation}
Then it follows from Lemma \ref{lp} that
\begin{equation}\label{lx}
 \|u_k\|^2_2=2J(u_k)-\langle J'(u_k),u_k\rangle\leq C+o_k(1)\|u_k\|^2_{\mathcal{H}_s}\leq C+o_k(1)\|u_k\|^2_{2}.  
\end{equation}
This means that $\{u_{k}\}$ is bounded in $\ell^2(V)$, and hence bounded in $\mathcal{H}_s$. 

We claim that
$$
u_{k} \nrightarrow 0,\quad  \text {in } \ell^{2}(V).
$$
In fact, if $u_{k} \rightarrow 0$ in $\ell^2(V)$, by Lemma \ref{lp}, we get that
$$\|u_k\|_{\mathcal{H}_s}\rightarrow 0.$$
Moreover, for $p \in(2,+\infty)$, by (\ref{ac}), we have $$\|u_{k}\|_p\rightarrow 0,\quad k\rightarrow+\infty.$$ Then it follows from (iii) of Proposition \ref{p1} that
\begin{equation*}
\int_{V} |B\left(u_{k}\right)|\,d \mu\leq C\int_V|u_k|^p\,d\mu \rightarrow 0,\quad \text{and}\quad \int_{V} |B^{\prime}\left(u_{k}\right) u_{k}|\,d \mu\leq C\int_V|u_k|^p\,d\mu \rightarrow 0.
\end{equation*}
Note that 
\begin{equation}\label{jj}
 \langle J^{\prime}\left(u_{k}\right), u_{k}\rangle=\|u_k\|^2_{\mathcal{H}_s}+\int_{V} A^{\prime}\left(u_{k}\right)u_{k}\,d \mu-\int_{V} B^{\prime}\left(u_{k}\right)u_{k}\,d \mu.   
\end{equation}
As a consequence, we get that
$$
\begin{aligned}
&\lim\limits_{k\rightarrow+\infty}\int_{V} A^{\prime}\left(u_{k}\right)u_{k}\,d \mu\\=&\lim\limits_{k\rightarrow+\infty}\left[\left\langle J^{\prime}\left(u_{k}\right), u_{k}\right\rangle+\int_{V} B^{\prime}\left(u_{k}\right) u_{k}\,d \mu-\left\|u_{k}\right\|^2_{\mathcal{H}_s}\right]\\=&0.
\end{aligned}
$$
Hence $A^{\prime}\left(u_{k}\right) u_{k} \rightarrow 0$ in $\ell^{1}(V)$. Since $A(t)$ is a convex and even function and satisfies $A(t)\geq 0$ for $t \in \mathbb{R}$. Then we have that $0 \leq A(t) \leq A^{\prime}(t) t$, which implies that $$A\left(u_{k}\right) \rightarrow 0,\quad \text{in}~\ell^1(V).$$ Therefore, we get that 
\begin{equation}\label{30}
\begin{aligned}
 c_0=\lim\limits_{k\rightarrow+\infty} J(u_k)=\lim\limits_{k\rightarrow+\infty}\left[\frac{1}{2}\|u_k\|_{\mathcal{H}_s}^{2}+ \int_{V} A(u_k)\,d\mu-\int_{V}B(u_k)\,d\mu\right]=0.
\end{aligned}
\end{equation}
which contradicts the fact that $c_0>0$. Hence there exist $a, b>0$ such that
\begin{equation}\label{20}
0<a \leq\left\|u_{k}\right\|_{2} \leq b.
\end{equation}
For any $u_{k}\in D(J)\backslash\{0\}$, there exists a unique $t_{k}>0$ such that $t_{k} u_{k} \in \mathcal{N}$, namely
$$
\int_{V}\left(\left|\nabla^s u_{k}\right|^{2}+h(x)u_{k}^{2}\right)\,d \mu-\int_{V} u_{k}^{2} \log \left(t_{k} u_{k}\right)^{2}\,d \mu=0.
$$
Moreover, by (\ref{29}), we have
$$
\left\langle J^{\prime}\left(u_{k}\right), u_{k}\right\rangle=\int_{V}\left(\left|\nabla^s u_{k}\right|^{2}+h(x)u^2_{k}\right)\,d \mu-\int_{V} u_{k}^{2} \log u_{k}^{2}\,d \mu=o_{k}(1).
$$
The above arguments imply that
$$
\left\|u_{k}\right\|_{2}^{2} \log t^2_{k}=o_k(1).
$$
Combined with (\ref{20}), we get that $t_{k} \rightarrow 1$  as $k\rightarrow+\infty.$
Hence, we obtain that
$$
\begin{aligned}
\inf _{u \in \mathcal{N}} J(u) \leq & J\left(t_{k} u_{k}\right)\\ = & J\left(t_{k} u_{k}\right)-\frac{1}{2}\langle J^{\prime}\left(t_ku_{k}\right), t_ku_{k}\rangle\\= &\frac{t^2_k}{2} \int_{V}u_{k}^{2}\,d \mu\\=&
t_{k}^{2}\left(J\left(u_{k}\right)+o_{k}(1)\left\|u_{k}\right\|_{\mathcal{H}_s}\right)\\=&t_{k}^{2} J\left(u_{k}\right)+o_{k}(1)\\\rightarrow& c_0,\quad k\rightarrow+\infty,
\end{aligned}
$$
which implies that $c\leq c_0$.   
\end{proof}

\
\

{\bf Proof of Theorem \ref{t1}:} By Lemma \ref{lh} and Lemma \ref{ld}, there exists a $(PS)_c$ sequence $\langle u_{k}\rangle \subset \mathcal{H}_s$ for $J$,
$$J(u_k)\rightarrow c,\quad \text{and}\quad J'(u_k)\rightarrow 0,\quad k\rightarrow+\infty. $$
Similar to (\ref{lx}), one gets that $\{u_{k}\}$ is bounded in $\mathcal{H}_s$ and \begin{equation}\label{66}
 u_{k} \nrightarrow 0,\quad  \text {in } \ell^{2}(V).   
\end{equation}

We claim that there exists $\delta>0$ such that
\begin{equation}\label{hw}
\liminf _{k\rightarrow\infty}\left\|u_k\right\|_{\infty} \geq \delta>0.
\end{equation}
Otherwise $\left\|u_k\right\|_{\infty}\rightarrow 0$ as $k\rightarrow+\infty$.
Then by Lions Lemma \ref{li}, we get that $u_k \rightarrow 0$ in $\ell^{p}(V)$ for $p>2$. Thus we have $\int_{V} B^{\prime}\left(u_{k}\right) u_{k}\,d \mu\rightarrow 0.$ Combined with (\ref{jj}), we get that
$$
\begin{aligned}
0=&\lim\limits_{k\rightarrow+\infty} \left[\langle J'(u_k),u_k\rangle + \int_{V} B^{\prime}\left(u_{k}\right) u_{k}\,d \mu\right]\\=& \lim\limits_{k\rightarrow+\infty}\left[\|u_k\|^2_{\mathcal{H}_s}+\int_{V} A^{\prime}\left(u_{k}\right) u_{k}\,d \mu\right].
\end{aligned}
$$
Since $A'(t)t\geq 0$ with $t\in\mathbb{R}$, we get that
$\lim\limits_{k\rightarrow+\infty}\|u_k\|_{\mathcal{H}_s}=0,$ and hence $\lim\limits_{k\rightarrow+\infty}\|u_k\|_{2}=0$, which contradicts $(\ref{66})$. The claim (\ref{hw}) is completed. 

It follows from (\ref{hw}) that there exists a sequence $\left\{y_k\right\} \subset V$ such that $\left|u_k(y_k)\right| \geq \frac{\delta}{2}.$
Let $$v_k(x)=u_k(x+y_k).$$ 
Since $h$ is $\mathbb{Z}^d$-periodic, $\mathcal{N}$ and $J$ and are invariant under translation, we obtain that $\left\{v_k\right\}$ is a bounded $(PS)_c$ sequence for $J$,
$$J\left(v_k\right) \rightarrow c,\qquad\text{and}\qquad J^{\prime}\left(v_k\right) \rightarrow 0, \quad k\rightarrow+\infty.$$
Thus there exists $v\in \mathcal{H}_s$ such that 
\begin{equation}\label{38}
 v_k \rightharpoonup v,\quad\text{in~}\mathcal{H}_s, \qquad\text{and}\qquad v_k\rightarrow v,\quad \text{pointwise~in~}V.   
\end{equation}
Clearly, $|v(0)|\geq \frac{\delta}{2}>0$, and hence $v\neq 0$. We prove that $v$ is a nontrivial critical point of $J$. In fact, since $J_2(v)=\int_{V}A(v)\,d\mu$ is lower semicontinuous and convex, it is also weakly lower semicontinuous. Hence $J_2(v)<+\infty$ and $v\in D(J)$. For any $\phi\in C_c(V)$, by (\ref{38}), we have that
$$\lim\limits_{k\rightarrow+\infty}\int_{V}(\nabla^s(v_k-v)\nabla^s \phi+(h(x)+1)(v_k-v)\phi)\,d \mu=0,$$
and hence
$$
\begin{aligned}
&\lim\limits_{k\rightarrow+\infty}\langle J'(v_k)-J'(v),\phi\rangle\\=&\lim\limits_{k\rightarrow+\infty}\left[\int_{V}(\nabla^s (v_k-v)\nabla^s \phi+(h(x)+1)(v_k-v)\phi)\,d \mu-\int_V(v_k-v)\phi\,d\mu-\int_{V}(v_k-v)\phi\log v_k^2\,d\mu\right]\\=&0.
\end{aligned}
$$
This means that for any $\phi\in C_c(V)$, $\langle J'(v),\phi\rangle=0$. Since $C_c(V)$ is dense in $D(J)$, we get that $v\in\mathcal{N}$. Finally, we prove that $J(v)=c$. Indeed, one has that
$$
\begin{aligned}
c \leq &J(v)\\=&J(v)-\frac{1}{2}\left\langle J^{\prime}(v), v\right\rangle\\=&\frac{1}{2}\int_{V}v^{2}\,d \mu \\
\leq & \frac{1}{2}\liminf _{k\rightarrow+\infty} \int_{V}v_k^{2}\,d \mu\\
\leq & \frac{1}{2}\limsup_{k\rightarrow+\infty} \int_{V}v_k^{2}\,d \mu\\=& \liminf _{k\rightarrow+\infty} \left(J(v_k)-\frac{1}{2}\langle J'(v_k),v_k)\right)\\=&c.
\end{aligned}
$$
Then $J(v)=c,$ and hence $v$ is a ground state solution to the equation (\ref{0.2}).

\qed

\section{Proof of Theorem \ref{t2}}
In this section, under the assumptions $(h_1)$ and $(h_3)$, we prove the existence of ground state sign-changing solutions to the equation (\ref{0.2}) by the Nehari manifold method and Miranda's theorem. We start our analysis with a few auxiliary lemmas that useful in our theorems.

\begin{lm}\label{l3}
 Let $u \in \mathcal{M}$. For $(\alpha,\beta) \in(0, +\infty) \times(0, +\infty)$ with $(\alpha,\beta) \neq(1,1)$, we have
$$
J(u) > J\left(\alpha u^{+}+\beta u^{-}\right).
$$

\end{lm}

\begin{proof}
Let 
$$f(x)=\left(1-x^{2}\right)+x^{2} \log x^{2}, \quad x \in(0,+\infty).$$ 
Then $f^{\prime}(x)=$ $2x\log x^{2}$. One gets easily that for $x \in(0,1)$, $f^{\prime}(x)<0$, and for $x \in(1,+\infty), f^{\prime}(x)>0$. Thus for $x \in(0,1) \cup(1,+\infty)$, 
\begin{equation}\label{26}
f(x)=\left(1-x^{2}\right)+x^{2} \log x^{2}>0.
\end{equation}

Let $u \in \mathcal{M}$ and $\alpha,\beta>0$. Note that $K(u)<0$, by Corollary \ref{co} and (\ref{26}), we obtain that
$$
\begin{aligned}
& J(u)-J\left(\alpha u^{+}+\beta u^{-}\right) \\
= & \frac{1}{2}\left(\|u\|_{\mathcal{H}_s}^{2}-\left\|\alpha u^{+}+\beta u^{-}\right\|_{\mathcal{H}_s}^{2}\right)-\frac{1}{2} \int_{V}\left[u^2 \log u^{2}-\left(\alpha u^{+}+\beta u^{-}\right)^{2} \log \left(\alpha u^{+}+\beta u^{-}\right)^{2}\right] \,d\mu.
\\
= & \frac{1-\alpha^2}{2}\left\|u^{+}\right\|_{\mathcal{H}_s}^{2}+\frac{1-\beta^2}{2}\left\|u^{-}\right\|_{\mathcal{H}_s}^{2}-\frac{(1-\alpha\beta)}{2}K(u)\\
& -\frac{1}{2} \int_{V}\left[\left(u^{+}\right)^{2} \log \left(u^{+}\right)^{2}-\left(\alpha u^{+}\right)^{2} \log \left(u^{+}\right)^{2}-\left(\alpha u^{+}\right)^{2} \log \alpha^2\right] \,d\mu\\
& -\frac{1}{2} \int_{V}\left[\left(u^{-}\right)^{2} \log \left(u^{-}\right)^{2}-\left(\beta u^{-}\right)^{2} \log \left(u^{-}\right)^{2}-\left(\beta u^{-}\right)^{2} \log \beta^2\right] \,d\mu\\
= & \frac{1-\alpha^2}{2}\left\langle J^{\prime}(u), u^{+}\right\rangle+\frac{1-\beta^2}{2}\left\langle J^{\prime}(u), u^{-}\right\rangle-\frac{1}{2}\left((1-\alpha\beta)-\frac{1-\alpha^2}{2}-\frac{1-\beta^2}{2}\right)K(u)\\
& +\frac{\left(1-\alpha^2\right)+\alpha^2 \log \alpha^2}{2} \int_{V}\left(u^{+}\right)^{2} \,d\mu+\frac{\left(1-\beta^2\right)+ \beta^2 \log \beta^2}{2} \int_{V}\left(u^{-}\right)^{2} \,d\mu
\\
= &-\frac{(\alpha-\beta)^2}{4} K(u)+\frac{\left(1-\alpha^2\right)+\alpha^2 \log \alpha^2}{2} \int_{V}\left(u^{+}\right)^{2} \,d\mu+\frac{\left(1-\beta^2\right)+ \beta^2 \log \beta^2}{2} \int_{V}\left(u^{-}\right)^{2} \,d\mu\\>&0.
\end{aligned}
$$
\end{proof}

\begin{lm}\label{4}
Let $u \in D(J)$ with $u^{ \pm} \neq 0$. Then there exists a unique positive number pair $\left(\alpha_{u}, \beta_{u}\right)$ such that $\alpha_{u} u^{+}+\beta_{u} u^{-} \in \mathcal{M}$.
\end{lm}

\begin{proof}
    
For $\alpha,\beta>0$, let
\begin{eqnarray*}
f(\alpha,\beta) = \left\langle J^{\prime}\left(\alpha u^{+}+\beta u^{-}\right),\alpha u^{+}\right\rangle,\qquad g(\alpha,\beta) =\left\langle J^{\prime}\left(\alpha u^{+}+\beta u^{-}\right),\beta u^{-}\right\rangle.  
\end{eqnarray*}
By Corollary \ref{co}, we get respectively that
\begin{equation}\label{2.7}
\begin{aligned}
f(\alpha,\beta)=&\langle J^{\prime}(\alpha u^{+}+\beta u^{-}),\alpha u^{+}\rangle\\=&(J^{\prime}\left(\alpha u^{+}\right),\alpha u^{+})-\frac{\alpha\beta}{2}K(u)\\=& \alpha^2\int_{V}\left(|\nabla^s u^+|^2 +h(x)(u^{+})^2\right)\,d \mu-\int_{V}\left(\alpha u^{+}\right)^{2} \log \left(\alpha u^{+}\right)^{2}\,d\mu-\frac{\alpha\beta}{2}K(u),
\end{aligned}
\end{equation}
and
\begin{equation}\label{2.8}
\begin{aligned}
g(\alpha,\beta) =&\langle J^{\prime}(\alpha u^{+}+\beta u^{-}),\beta u^{-}\rangle\\=&(J^{\prime}\left(\beta u^{-}\right),\beta u^{-})-\frac{\alpha\beta}{2}K(u)\\=& \beta^2\int_{V}\left(|\nabla^s u^{-}|^2 +h(x)(u^{-})^2\right)\,d \mu-\int_{V}\left(\beta u^{-}\right)^{2} \log \left(\beta u^{-}\right)^{2}\,d\mu-\frac{\alpha\beta}{2}K(u).
\end{aligned}
\end{equation}
Let $\beta=\alpha$ in (\ref{2.7}) and (\ref{2.8}), then
$$
\begin{aligned}
f(\alpha,\alpha) =& \alpha^2\int_{V}\left(|\nabla^s u^+|^2 +h(x)(u^{+})^2\right)\,d \mu-\alpha^2\log \alpha^2\int_{V}\left(u^{+}\right)^{2}\,d\mu-\frac{\alpha^2}{2}K(u)-\alpha^2\int_{V}\left( u^{+}\right)^{2} \log \left( u^{+}\right)^{2}\,d\mu,
\end{aligned}
$$
and
$$
\begin{aligned}
g(\alpha,\alpha) =& \alpha^2\int_{V}\left(|\nabla^s u^{-}|^2 +h(x)(u^{-})^2\right)\,d \mu-\alpha^2\log \alpha^2\int_{V}\left(u^{-}\right)^{2}\,d\mu-\frac{\alpha^2}{2}K(u)-\alpha^2\int_{V}\left( u^{-}\right)^{2} \log \left( u^{-}\right)^{2}\,d\mu.
\end{aligned}
$$
Note that $K(u)< 0$. Hence for $\alpha>0$ sufficiently small, we get that $f(\alpha,\alpha)>0$ and $g(\alpha,\alpha)>0$. For $\alpha>0$ sufficiently large, we have that $f(\alpha,\alpha)<0$ and $g(\alpha,\alpha)<0$. As a consequence, there exist $r$ and $R$ with $0<r<R$ such that
\begin{equation}\label{0.5}
f(r, r)>0,\quad g(r, r)>0,\quad\text{and}\quad f(R, R)<0,\quad g(R, R)<0.
\end{equation}
Then for $\beta\in[r, R],$ it follows from (\ref{2.7}) and (\ref{0.5}) that
$$
\begin{aligned}
&f(r, \beta)\geq f(r,r)>0, \quad\text{and}\quad f(R, \beta)\leq f(R,R)<0.
\end{aligned}
$$
For $\alpha\in[r, R],$ by (\ref{2.8}) and (\ref{0.5}), we get that
$$
\begin{aligned}
& g(\alpha, r)\geq g(r,r)>0,\quad \text{and}\quad g(\alpha, R)\leq g(R,R)<0.
\end{aligned}
$$
By the Miranda's theorem \cite{K}, there exist $\alpha_{u}, \beta_{u} \in(r, R)$ such that $f\left(\alpha_{u}, \beta_{u}\right)=g\left(\alpha_{u}, \beta_{u}\right)=0$, which implies that $\alpha_{u} u^{+}+\beta_{u} u^{-} \in \mathcal{M}$.

At last, we prove the uniqueness of the positive number pair $\left(\alpha_{u}, \beta_{u}\right)$. By contradiction, suppose there exists $\left(\alpha_i, \beta_{i}\right)$ such that $\alpha_{i} u^{+}+\beta_{i} u^{-} \in \mathcal{M}$ for $i=1, 2$, where $\alpha_1\neq \alpha_2$ and $\beta_1\neq \beta_2$. 

Let $\alpha=\frac{\alpha_2}{\alpha_1}\neq1$ and $\beta=\frac{\beta_2}{\beta_1}\neq1$. By Lemma \ref{l3}, we get that
\begin{equation*}\label{3.1}
J(\alpha_2u^++\beta_2u^-) =J \left(\alpha(\alpha_1u^+)+\beta(\beta_1u^-)\right)<J (\alpha_1u^++\beta_1u^-),
\end{equation*}
and
\begin{equation*}\label{3.2}
  J(\alpha_1u^++\beta_1u^-) =J \left(\frac{1}{\alpha}(\alpha_2u^+)+\frac{1}{\beta}(\beta_2u^-)\right)<J (\alpha_2u^++\beta_2u^-).  
\end{equation*}
This is a contradiction. The proof is completed.

\end{proof}

\begin{lm}\label{l5}
Let $u \in D(J)$ with $u^{ \pm} \neq 0$ such that $\langle J^{\prime}(u),u^{ \pm}\rangle \leq 0$. Then the unique pair $\left(\alpha_{u}, \beta_{u}\right)$ obtained in Lemma \ref{4} satisfies $\alpha_{u}, \beta_{u} \in(0,1]$. In particular, the $"="$ holds if and only if $\alpha_u=\beta_u=1.$
\end{lm}

\begin{proof}
By Lemma \ref{4}, there exists a unique positive number pair $(\alpha_u,\beta_u)$ such that $\alpha_uu^++\beta_uu^-\in \mathcal{M}$. Without loss of generality, we assume that $0<\beta_{u} \leq \alpha_{u}$. By Corollary \ref{co}, we get that
\begin{equation}\label{3.4}
\begin{aligned}
&\langle J^{\prime}(\alpha_uu^++\beta_uu^-),\alpha_uu^{+}\rangle\\ &=\alpha_u^2\int_{V}\left(|\nabla^s u^+|^2 +h(x)(u^{+})^2\right)\,d \mu-\int_{V}\left(\alpha_u u^{+}\right)^{2} \log \left(\alpha_u u^{+}\right)^{2}\,d\mu-\frac{\alpha_u\beta_u}{2}K(u).
\end{aligned}
\end{equation}
Since $\alpha_{u} u^{+}+\beta_{u} u^{-} \in \mathcal{M}$ and $K(u)<0$, we get that
\begin{equation}\label{3.3}
\begin{aligned}
\int_{V}\left(\alpha_u u^{+}\right)^{2} \log \left(\alpha_u u^{+}\right)^{2}\,d\mu=&\alpha_u^2\int_{V}\left(|\nabla^s u^+|^2 +h(x)(u^{+})^2\right)\,d \mu-\frac{\alpha_u\beta_u}{2}K(u)\\ \leq &\alpha_u^2\int_{V}\left(|\nabla^s u^+|^2 +h(x)(u^{+})^2\right)\,d \mu-\frac{\alpha^2_u}{2}K(u).
\end{aligned}    
\end{equation}

Let $\alpha_u=\beta_u=1$ in (\ref{3.4}), we have that
$$
\begin{aligned}
\langle J^{\prime}(u),u^{+}\rangle=\int_{V}\left(|\nabla^s u^+|^2 +h(x)(u^{+})^2\right)\,d \mu-\int_{V}\left(u^{+}\right)^{2} \log \left(u^{+}\right)^{2}\,d\mu-\frac{1}{2}K(u).
\end{aligned}
$$
Then it follows from the fact $\langle J^{\prime}(u),u^{+}\rangle\leq 0$ that
$$
\begin{aligned}
\int_{V}\left(u^{+}\right)^{2} \log \left(u^{+}\right)^{2}\,d\mu\geq \int_{V}\left(|\nabla^s u^+|^2 +h(x)(u^{+})^2\right)\,d \mu-\frac{1}{2}K(u).
\end{aligned}
$$
Multiplying the above inequality by $-\alpha_u^2$, then
\begin{equation}\label{3.5}
 \begin{aligned}
-\alpha_u^2\int_{V}\left(u^{+}\right)^{2} \log \left(u^{+}\right)^{2}\,d\mu\leq -\alpha_u^2\int_{V}\left(|\nabla^s u^+|^2 +h(x)(u^{+})^2\right)\,d \mu+\frac{\alpha_u^2}{2}K(u).
\end{aligned}   
\end{equation}
By (\ref{3.3}) and (\ref{3.5}), we get that
$$
\begin{aligned}
\alpha_u^2\log \alpha_u^{2}\int_{V}\left(u^{+}\right)^{2}\,d\mu\leq 0.
\end{aligned} 
$$
This implies that $0<\alpha_u\leq 1,$ and hence $0<\beta_u\leq\alpha_u\leq 1.$
\end{proof}

Now we prove that the minimizer of $J$ on $\mathcal{M}$ can be achieved.

\begin{lm}\label{l9}
Let $(h_1)$ and $(h_3)$ hold. Then $m>0$ is achieved.
\end{lm}

\begin{proof} 
Let $\{u_k\}\subset\mathcal{M}$ be a minimizing sequence satisfying
$$\lim\limits_{k\rightarrow\infty}J(u_k)=m.$$
Since $\{u_k\}\subset\mathcal{M}$, we have that 
$$
\begin{aligned}
\lim _{k \rightarrow\infty} J\left(u_{k}\right) & =\lim _{k \rightarrow\infty}\left[J\left(u_{k}\right)-\frac{1}{2} \langle J^{\prime}\left(u_{k}\right), u_{k}\rangle\right] \\
& =\frac{1}{2}\lim _{k \rightarrow\infty}\int_V |u_k|^2\,d\mu \\
& =m.
\end{aligned}
$$
This means that $\{u_{k}\}$ is bounded in $\ell^2(V)$. By (\ref{ac}), for any $q>2$, there exists $C_{q}>0$ such that
$$
\begin{aligned}
\int_{V} u_{k}^{2} \log u_{k}^{2} d \mu & \leq \int_{V}\left(u_{k}^{2} \log u_{k}^{2}\right)^{+} d \mu\\& \leq C_{q} \int_{V}\left|u_{k}\right|^{q} d \mu\\
& \leq C_{q}\left\|u_{k}\right\|_{2}^{q}.
\end{aligned}
$$
Then we get that 
\begin{equation}\label{90}
\left\|u_{k}\right\|_{\mathcal{H}_s}^{2}=\int_{V} u_{k}^{2} \log u_{k}^{2} d \mu+\left\|u_{k}\right\|_{2}^{2} \leq C_{q}\left\|u_{k}\right\|_{2}^{q}+\left\|u_{k}\right\|_{2}^{2}.
\end{equation}
Hence $\left\{u_{k}\right\}$ is bounded in $\mathcal{H}_{s}$. Then by Lemma \ref{l2}, there exists $u \in \mathcal{H}_{s}$ such that, passing to a subsequence if necessary, 
\begin{equation}\label{91}
\begin{cases}u_{k} \rightharpoonup u, & \text { weakly in } \mathcal{H}_{s}, \\ u_{k} \rightarrow u, & \text { pointwise in } V, \\ u_{k} \rightarrow u, & \text { strongly in } \ell^{q}(V), ~q \in[2,\infty],\end{cases}
\end{equation}
By the fact $u_k^{\pm}\neq 0$ and (\ref{91}), we have that 
$$
\begin{aligned}
0< &C\leq \|u_k^{\pm}\|^2_{\mathcal{H}_s}\\ \leq &\int_V(u_k^{\pm})^2\,d\mu+\int_{V}\left(u_k^{\pm}\right)^{2} |\log \left(u_k^{\pm}\right)^{2}|\,d\mu\\ \leq &\int_V(u_k^{\pm})^2\,d\mu+C_{q} \int_{V}\left|u^{\pm}_{k}\right|^{q} d \mu\\ \leq &\left\|u^{\pm}_{k}\right\|_{2}^{2}+C_{q}\left\|u^{\pm}_{k}\right\|_{2}^{q}\\= & \left\|u^{\pm}\right\|_{2}^{2}+C_{q}\left\|u^{\pm}\right\|_{2}^{q}+o_k(1),
\end{aligned}
$$
which implies that $u^{\pm}\neq 0$, and hence 
$$
\begin{aligned}
m&=\lim _{k \rightarrow\infty} J\left(u_{k}\right)\\ & =\lim _{k \rightarrow\infty}\left[J\left(u_{k}\right)-\frac{1}{2}\langle J^{\prime}\left(u_{k}\right), u_{k}\rangle\right] \\
& =\frac{1}{2}\lim _{k \rightarrow\infty}\int_V |u_k|^2\,d\mu \\
& =\frac{1}{2}\int_V (|u^+|^2+|u^-|^2)\,d\mu\\&>0.
\end{aligned}
$$

Note that
\begin{equation}\label{3.8}
\langle J'(u_k),u_k^{\pm}\rangle=\|u_k^{\pm}\|_{\mathcal{H}_s}^{2}-\frac{1}{2}K(u_k)-\int_{V}(u_k^{\pm})^2\,d\mu-\int_{V}\left(u_k^{\pm}\right)^{2} \log \left(u_k^{\pm}\right)^{2}\,d\mu=0.
\end{equation}
By (\ref{91}), (\ref{3.8}), weak lower semicontinuity of norm, Fatou's lemma and Lebesgue dominated theorem, we get that
$$
\begin{aligned}
&\|u^{\pm}\|_{\mathcal{H}_s}^{2}-\frac{1}{2}K(u)-\int_{V}(u^{\pm})^2\,d\mu-\int_{V}\left((u^{\pm})^{2} \log \left(u^{\pm}\right)^{2}\right)^-\,d\mu \\
\leq & \liminf _{k \rightarrow\infty}\left[\|u_k^{\pm}\|_{\mathcal{H}_s}^{2}-\frac{1}{2}K(u_k)-\int_{V}(u_k^{\pm})^2\,d\mu-\int_{V}\left((u_k^{\pm})^{2} \log \left(u_k^{\pm}\right)^{2}\right)^-\,d\mu\right] \\
= & \liminf _{k \rightarrow\infty} \int_{V}\left((u_k^{\pm})^{2} \log \left(u_k^{\pm}\right)^{2}\right)^+\,d\mu \\
= & \int_{V}\left((u^{\pm})^{2} \log \left(u^{\pm}\right)^{2}\right)^+\,d\mu,
\end{aligned}
$$
and hence
\begin{equation*}
\langle J^{\prime}\left(u\right), u^{\pm}\rangle=\|u^{\pm}\|_{\mathcal{H}_s}^{2}-\frac{1}{2}K(u)-\int_{V}(u^{\pm})^2\,d\mu-\int_{V}|u^{\pm}|^{2} \log \left(u^{\pm}\right)^{2}\,d\mu\leq0.
\end{equation*}
By Lemma \ref{l5}, there exist two constants $\alpha, \beta\in(0,1]$ such that $\tilde{u}=\alpha u^{+}+\beta u^{-} \in \mathcal{M}$. Then
$$
\begin{aligned}
m \leq &J(\widetilde{u})=J\left(\widetilde{u}\right)-\frac{1}{2} \langle J^{\prime}\left(\widetilde{u}\right), \widetilde{u}\rangle\\=&\frac{1}{2}\int_V |\alpha u^{+}+\beta u^{-}|^2\,d\mu\\ \leq& \frac{1}{2}\int_V |u|^2\,d\mu\\ =&
\frac{1}{2}\lim\limits_{k \rightarrow\infty}\int_V |u_k|^2\,d\mu\\ =&
\lim\limits_{k \rightarrow\infty}\left[J\left(u_{k}\right)-\frac{1}{2}\langle J^{\prime}\left(u_{k}\right), u_{k}\rangle\right] \\=& m.
\end{aligned}
$$
This implies that $\alpha=\beta=1$. Hence $u \in \mathcal{M}$ and  $J\left(u\right)=m>0.$

\end{proof}

The following lemma completes the proof of Theorem \ref{t1}.

\begin{lm}\label{l8}
If $u \in \mathcal{M}$ with $J(u)=m$, then $u$ is a sign-changing solution of the equation (\ref{0.2}). 
\end{lm}

\begin{proof}
By contradiction, we assume  that $u \in \mathcal{M}$ with $J(u)=m$, but $u$ is not a solution of the equation (\ref{0.2}). Then there exist a function $\phi \in D(J)$ such that
$$
\int_{V}(\nabla^su\nabla^s \phi+h(x) u \phi)\,d \mu-\int_{V}u\phi\log u^2\,d\mu\leq -1.
$$
Since $J$ is a $C^1$ functional on $D(J)$, for some $\varepsilon>0$,
$$
\left\langle J^{\prime}\left(\alpha u^{+}+\beta u^{-}+\sigma \phi\right), \phi\right\rangle\leq-\frac{1}{2},\quad |\alpha-1|+|\beta-1|+|\sigma| \leq \varepsilon.
$$

We consider the functional $J\left(\alpha u^{+}+\beta u^{-}+\varepsilon \eta(\alpha,\beta) \phi\right)$, where $\eta$ is a cutoff function satisfying
$$
\eta(\alpha,\beta)= 
\begin{cases}1, & |\alpha-1| \leq \frac{1}{2} \varepsilon, |\beta-1| \leq \frac{1}{2} \varepsilon, \\ 0, & |\alpha-1| \geq \varepsilon \text { or }|\beta-1| \geq \varepsilon.
\end{cases}
$$
If $|\alpha-1| \leq \varepsilon$ and $|\beta-1| \leq \varepsilon$, we have
$$
\begin{aligned}
J\left(\alpha u^{+}+\beta u^{-}+\varepsilon \eta(\alpha,\beta) \phi\right) &  =J\left(\alpha u^{+}+\beta u^{-}\right)+\int_0^{1} \left\langle J^{\prime}\left(\alpha u^{+}+\beta u^{-}+\sigma \varepsilon \eta(\alpha,\beta) \phi\right),\varepsilon \eta(\alpha,\beta) \phi\right\rangle\,d \sigma \\
& \leq J\left(\alpha u^{+}+\beta u^{-}\right)-\frac{1}{2} \varepsilon \eta(\alpha,\beta).
\end{aligned}
$$
If $|\alpha-1| \geq \varepsilon$ or $|\beta-1| \geq \varepsilon,\,  \eta(\alpha,\beta)=0$, the above estimate is obvious. Since $u \in \mathcal{M}$, for $(\alpha,\beta) \neq(1,1)$, by Lemma \ref{l3}, we have 
$$
J\left(\alpha u^{+}+\beta u^{-}+\varepsilon \eta(\alpha,\beta) \phi\right) \leq J\left(\alpha u^{+}+\beta u^{-}\right)<J(u), \quad (\alpha,\beta) \neq(1,1).
$$
For $(\alpha,\beta)=(1,1)$,
$$
J\left(\alpha u^{+}+\beta u^{-}+\varepsilon \eta(\alpha,\beta) \phi\right) \leq J\left(\alpha u^{+}+\beta u^{-}\right)-\frac{1}{2} \varepsilon \eta(1,1)=J(u)-\frac{1}{2} \varepsilon.
$$
In any case, we have $J\left(\alpha u^{+}+\beta u^{-}+\varepsilon \eta(\alpha,\beta) \phi\right)<J(u)=m$. In particular, for $0<\delta<1-\varepsilon$,
$$
\sup _{\delta \leq\alpha,\beta\leq 2-\delta} J\left(\alpha u^{+}+\beta u^{-}+\varepsilon \eta(\alpha,\beta) \phi\right)=\tilde{m}<m.
$$
Let $v=\alpha u^{+}+\beta u^{-}+\varepsilon \eta(\alpha,\beta) \phi$. Define
$$
F(\alpha,\beta)=\left\langle J^{\prime}(v), v^{+}\right\rangle,\qquad G(\alpha,\beta)= \left\langle J^{\prime}(v), v^{-}\right\rangle.
$$
Since $u\in\mathcal{M}$, we have
\begin{equation}\label{3.7}
\begin{aligned}
\int_{V}\left(u^{+}\right)^{2} \log \left(u^{+}\right)^{2}\,d\mu= &\|u^{+}\|_{\mathcal{H}_s}^{2}-\frac{1}{2}K(u)-\int_{V}(u^+)^2\,d\mu,
\end{aligned}
\end{equation}
and 
\begin{equation}\label{4.7}
\begin{aligned}
\int_{V}\left(u^{-}\right)^{2} \log \left(u^{-}\right)^{2}\,d\mu= &\|u^{-}\|_{\mathcal{H}_s}^{2}-\frac{1}{2}K(u)-\int_{V}(u^-)^2\,d\mu.
\end{aligned}
\end{equation}
By the definition of $\eta$, for $\alpha=\delta<1-\varepsilon$ and $\beta \in(\delta, 2-\delta)$, we have $\eta(\alpha,\beta)=0$ and $\alpha<\beta$. Hence by (\ref{3.7}), we get that
$$
\begin{aligned}
F(\delta,\beta)=&\langle J^{\prime}(\delta u^++\beta u^{-}),\delta u^{+}\rangle\\=&\delta^{2}\left[\|u^{+}\|_{\mathcal{H}_s}^{2}-\int_{V}(u^+)^2\,d\mu-\int_{V}\left(u^{+}\right)^{2} \log \left(u^{+}\right)^{2}\,d\mu-\log \delta^2\int_{V}\left(u^{+}\right)^{2}\,d\mu\right]-\frac{\delta\beta}{2}K(u)\\ \geq &\delta^{2}\left[\|u^{+}\|_{\mathcal{H}_s}^{2}-\int_{V}(u^+)^2\,d\mu-\int_{V}\left(u^{+}\right)^{2} \log \left(u^{+}\right)^{2}\,d\mu-\log \delta^2\int_{V}\left(u^{+}\right)^{2}\,d\mu-\frac{1}{2}K(u)\right]\\=&-\delta^2\log \delta^2\int_{V}\left(u^{+}\right)^{2}\,d\mu\\>&0.
\end{aligned}
$$
For $\alpha=2-\delta>1+\varepsilon$ and $\beta \in(\delta, 2-\delta)$, we have $\eta(\alpha,\beta)=0$ and $\alpha>\beta$. Similarly, we get that
$$
\begin{aligned}
F(2-\delta,\beta)=&\langle J^{\prime}((2-\delta)u^++\beta u^{-}),(2-\delta)u^{+}\rangle\\=&(2-\delta)^{2}\left[\|u^{+}\|_{\mathcal{H}_s}^{2}-\int_{V}(u^+)^2\,d\mu-\int_{V}\left(u^{+}\right)^{2} \log \left(u^{+}\right)^{2}\,d\mu-\log (2-\delta)^2\int_{V}\left(u^{+}\right)^{2}\,d\mu\right]\\&-\frac{(2-\delta)\beta}{2}K(u)\\ \leq &(2-\delta)^{2}\left[\|u^{+}\|_{\mathcal{H}_s}^{2}-\int_{V}(u^+)^2\,d\mu-\int_{V}\left(u^{+}\right)^{2} \log \left(u^{+}\right)^{2}\,d\mu-\log (2-\delta)^2\int_{V}\left(u^{+}\right)^{2}\,d\mu \right]\\&-\frac{(2-\delta)^{2}}{2}K(u)
\\=&-(2-\delta)^2\log (2-\delta)^2\int_{V}\left(u^{+}\right)^{2}\,d\mu\\<& 0.
\end{aligned}
$$
Namely,
$$
F(\delta, \beta)>0,\qquad F(2-\delta, \beta)<0,\quad \beta\in(\delta, 2-\delta).
$$
By (\ref{4.7}) and similar arguments as above, we obtain that
$$
G(\alpha, \delta)>0,\qquad G(\alpha, 2-\delta)<0,\quad \alpha\in(\delta, 2-\delta).
$$
By the Miranda's theorem \cite{K}, there exists $\left(\alpha_0, \beta_0\right) \in(\delta, 2-\delta) \times(\delta, 2-\delta)$ such that $$\widetilde{u}=\alpha_0 u^{+}+\beta_0 u^{-}+\varepsilon \eta\left(\alpha_0, \beta_0\right) \phi \in\mathcal{M},$$ and $J(\widetilde{u})<m.$  This contradicts the definition of $m$.
We complete the proof.

\end{proof}

\
\

{ \bf Declaration}
The author declares that there are no conflict of interest regarding the publication of this paper.

\
\

{ \bf Data availability}
No data was used for the research described in this paper.

\
\

{\bf Appendix}

We give an example to show that there exists $u\in W^{s,2}(\mathbb{Z}^d)$ such that $\int_{\mathbb{Z}^d}u^2\log u^2\,d\mu=-\infty.$

We first recall the  fact that
$$
\sum_{n=2}^{\infty} \frac{1}{n(\log n)^{p}} 
\begin{cases}<+\infty, & \text { if } p>1, \\ =+\infty, & \text { if } 0 \leq p \leq 1.
\end{cases}
$$

Denote $e=(1,0,\cdots,0)\in\mathbb{Z}^d.$ Let
$$
u(x)= 
\begin{cases}
(|x|^{\frac{1}{2}} \log |x|)^{-1}, & x=ne, n\geq 3, \\ 0, & \text{rest}.
\end{cases}
$$
Then $u\in W^{s,2}(\mathbb{Z}^d)$ but $\int_{\mathbb{Z}^d}u^2\log u^2\,d\mu=-\infty.$ Indeed, a direct calculation yields that
$$
\int_{\mathbb{Z}^d} |u|^{2}\,d \mu=\sum_{x \in \mathbb{Z}^d} |u(x)|^{2}=\sum_{n=3}^{+\infty} \frac{1}{n(\log n)^{2}}<+\infty.
$$
By Lemma \ref{lb}, one gets that $u\in W^{s,2}(\mathbb{Z}^d)$. But for the term $\int_{\mathbb{Z}^d} u^2\log u^2\,d\mu$, we have
 $$
 \begin{aligned}
  \int_{\mathbb{Z}^d}u^2|\log u^2|\,d\mu=& \sum\limits_{x\in\mathbb{Z}^d}|u(x)|^2 |\log u(x)|^2\\ =&-\left(\sum_{n=3}^{+\infty}\frac{1}{n\log n}+\sum\limits_{n=3}^{+\infty}\frac{2\log(\log n)}{n(\log n)^2}\right)\\=&-\infty.
 \end{aligned}
 $$
\qed


\begin{thebibliography}{99}



\bibitem{AD} C.O. Alves, D.C. de Morais Filho, Existence and concentration of positive solutions for a Schr\"{o}dinger logarithmic equation. Z. Angew. Math. Phys. 69 (2018), no. 6, Paper No. 144, 22 pp.

\bibitem{AJ1} C.O. Alves, C. Ji,  Existence and concentration of positive solutions for a logarithmic Schr\"{o}dinger equation via penalization method. Calc. Var. Partial Differential Equations 59 (2020), no. 1, Paper No. 21, 27 pp.

\bibitem{A} A.H. Ardila, Existence and stability of standing waves for nonlinear fractional Schr\"{o}dinger equation with logarithmic nonlinearity. Nonlinear Anal. 155 (2017), 52-64.

\bibitem{CR0} O. Ciaurri, L. Roncal, P. Stinga, J. Torrea and J. Varona, Nonlocal discrete diffusion equations and the fractional discrete Laplacian, regularity and applications. Adv. Math. 330 (2018) 688-738.



\bibitem{CWY} X. Chang, R. Wang and D. Yan, Ground states for logarithmic Schr\"{o}dinger equations on locally finite graphs. J. Geom. Anal. 33 (2023), no.7, Paper No. 211.

\bibitem{CR} X. Chang, V.D. R$\breve{a}$dulescu,  R. Wang and D. Yan, Convergence of least energy sign-changing solutions for logarithmic Schr\"{o}dinger equations on locally finite graphs. Commun. Nonlinear Sci. Numer. Simul. 125 (2023), Paper No. 107418.


\bibitem{DS} P. d'Avenia, M. Squassina and M. Zenari, Fractional logarithmic Schr\"{o}dinger equations. Math. Methods Appl. Sci. 38 (2015), no. 18, 5207-5216.

\bibitem{FT} W. Feng, X. Tang and L. Zhang, Existence of a positive bound state solution for logarithmic Schr\"{o}dinger equation. J. Math. Anal. Appl. 531 (2024), no. 2, Paper No. 127861.




\bibitem {GLY} A. Grigor'yan, Y. Lin and Y.Yang,  Existence of positive solutions to
some nonlinear equations on locally finite graphs. Sci. China Math. 60 (2017), 1311-1324.

\bibitem{HL0} F. Han, L. Wang, Positive solutions to discrete harmonic functions in unbounded cylinders. J. Korean Math. Soc. 61 (2024), no. 2, 377-393. 

\bibitem{HJ} Z. He, C. Ji, Existence and multiplicity of solutions for the logarithmic Schr\"{o}dinger equation with a potential on lattice graphs. arXiv:2403.15866.

\bibitem{HL1} S. He, X. Liu, Multiple solutions for a class of fractional logarithmic Schr\"{o}dinger equations. 
Partial Differ. Equ. Appl. 2 (2021), no. 6, Paper No. 70, 30 pp.

\bibitem {HLW} B. Hua, R. Li and L. Wang, A class of semilinear elliptic equations on groups of polynomial growth. J. Differential Equations 363 (2023), 327-349.




\bibitem{J0} C. Ji, Multi-bump type nodal solutions for a logarithmic Schr\"{o}dinger equation with deepening potential well. Z. Angew. Math. Phys. 72 (2021), no. 2, Paper No. 70, 25 pp.

\bibitem{JS} C. Ji, A. Szulkin, 
A logarithmic Schr\"{o}dinger equation with asymptotic conditions on the potential. 
J. Math. Anal. Appl. 437 (2016), no. 1, 241-254.

\bibitem{JX} C. Ji, Y. Xue, Existence and concentration of positive solutions for a fractional logarithmic Schr\"{o}dinger equation. Differential Integral Equations 35 (2022), no. 11-12, 677-704.

\bibitem{K} W. Kulpa. The Poinca\'{e}-Miranda theorem. Am. Math. Mon. 104 (1997), 545-550.


\bibitem{LP} Q. Li, S. Peng and W. Shuai, On fractional logarithmic Schr\"{o}dinger equations. Adv. Nonlinear Stud. 22 (2022), no. 1, 41-66. 

\bibitem{LM} C. Lizama and M. Murillo-Arcila, On a connection between the $N$-dimensional fractional Laplacian and 1-D operators on lattices. J. Math. Anal. Appl. 511 (2022) 126051.



\bibitem{LR} C. Lizama, L. Roncal, H\"{o}lder-Lebesgue regularity and almost periodicity for semidiscrete equations with a fractional Laplacian. Discrete Contin. Dyn. Syst. Ser. A 38(3) (2018) 1365-1403.
 



\bibitem{OZ} X. Ou, X. Zhang,  Ground-state sign-changing homoclinic solutions for a discrete nonlinear $p$-Laplacian equation with logarithmic nonlinearity. Bound. Value Probl. (2024), Paper No. 6.



\bibitem{SY} M. Shao, Y. Yang and L. Zhao, 
Multiplicity and limit of solutions for logarithmic Schr\"{o}dinger equations on graphs.  
J. Math. Phys. 65 (2024), no. 4, Paper No. 041508.


\bibitem{SS} M. Squassina, A. Szulkin, Multiple solutions to logarithmic Schr\"{o}dinger equations with periodic potential. Calc. Var. Partial Differential Equations 54 (2015), no. 1, 585-597.


\bibitem{S0}  W. Shuai, Existence and multiplicity of solutions for logarithmic Schr\"{o}dinger equations with potential. J. Math. Phys. 62 (2021), no. 5, Paper No. 051501.


\bibitem{S}  A. Szulkin, Minimax principles for lower semicontinuous functions and applications to nonlinear boundary value problems. Ann. Inst. H. Poincar\'{e} Anal. Non Lin\'{e}aire 3 (1986), no. 2, 77-109.

\bibitem{WF} L. Wang, S. Feng,  and K. Cheng, Existence and concentration of positive solutions for a fractional Schr\"{o}dinger logarithmic equation. Complex Var. Elliptic Equ. 69 (2024), no. 2, 317-348.

 \bibitem{W0} J. Wang, Eigenvalue estimates for the fractional Laplacian on lattice subgraphs. arXiv: 2303.15766.
 
\bibitem{W1} L. Wang, Solutions to discrete fractional Sch\"{o}dinger equations. arXiv: 2308.09879.

\bibitem{W2} L. Wang, The ground state solutions of discrete nonlinear Schr\"{o}dinger equations with Hardy weights. Mediterr. J. Math. 21 (2024), no. 3, Paper No.78.

\bibitem{W3} L. Wang, Solutions to discrete nonlinear Kirchhoff-Choquard equations. Bull. Malays. Math. Sci. Soc. 47 (2024), no. 5, Paper No. 138.

\bibitem{XZ} M. Xiang, B. Zhang, Homoclinic solutions for fractional discrete Laplacian equations. Nonlinear Anal. 198 (2020), 111886, 15 pp.

\bibitem{ZLY}M. Zhang, Y. Lin and Y. Yang,
Fractional Laplace operator and related Schr\"{o}dinger equations on locally finite graphs. arXiv:2408.02902

\bibitem{ZZ} N. Zhang, L. Zhao,   Convergence of ground state solutions for nonlinear Schr\"{o}dinger equations on graphs. Sci. China Math. 61 (2018), 1481-1494.



\end{thebibliography}
\end{document}